\documentclass{amsart}
\usepackage{amssymb,amsmath,amsthm,latexsym}
\usepackage{footnote}
\usepackage{xcolor}
\usepackage{graphicx}
\usepackage{float}
\usepackage[para]{threeparttable}
\usepackage{color,soul}

\newtheorem{thm}{Theorem}[section]

\newtheorem{lem}[thm]{Lemma}

\newtheorem{defi}[thm]{Definition}
\newtheorem{notat}[thm]{Notation}
{ \theoremstyle{remark} }

\begin{document}

\title[A Generalization of the Hausdorff Dimension Theorem for Fractals]
{A Generalization of the Hausdorff Dimension Theorem for Fractals}
\date{Jun 5, 2021}
\maketitle
\begin{center}
\author {Mohsen Soltanifar\footnote{Biostatistics Division, Dalla Lana School of Public Health, University of Toronto, Toronto, ON, Canada\\
e-mail: mohsen.soltanifar@alum.utoronto.ca \\
ORCID: https://orcid.org/0000-0002-5989-0082}}
\end{center}

\begin{abstract}
How many fractals exist in nature or the virtual world? In this note, we partially answer the second question using Mandelbrot’s fundamental definition of fractals and their quantities of the Hausdorff dimension and Lebesgue measure. We prove the existence of  beth- two of virtual fractals with a Hausdorff dimension of a bi-variate function of them and the given Lebesgue measure. The question remains unanswered for other  fractal dimensions.
\end{abstract}

																												

 
\section{Introduction}
Benoit Mandelbrot (1924-2010) coined the term \emph{fractal} and its dimension in his 1975 essay on the quest to present a mathematical model for self-similar, fractured geometric shapes in nature \cite{R1}. Both nature and the virtual world have prominent examples of fractals: Some natural examples include snowflakes, clouds and mountain ranges, while some virtual instances are the middle third Cantor set, the Sierpinski triangle and the Vicsek fractal. A key element in the definition of the term \emph{fractal} is its dimension  usually indexed by the Hausdorff dimension. For instance, the usual Vicsek has the Hausdorff dimesnion equal to $\log_{3}(5).$ While each fractal has its Hausdorff dimension as a unique number in $(0,+\infty)$, there has been little evidence on the inverse existence statement. Given that constructive existence theorems play a key role in a wide spectrum of mathematical fields and computer science \cite{RR2,RR3}, one may ask: {\em Is there any fractal for given Hausdorff dimension ?}  And, in case of affirmative response, {\em What are the features of the set of such fractals?} To the author's best knowledge, the only available existence theorem of fractals in the literature is limited to the case of the fractals with Lebesgue measure 0 (thin fractals) \cite{R2} and its very minor version in $\mathbb{R}$\cite{RR1}. However, on one hand, there are  fractals in the Euclidean spaces with a positive Lebesgue measure (fat fractals) applied to model real physical systems \cite{RR4,RR5}. Examples include the fat Cantor set with Hausdorff dimension of 1, the fat Vicsek fractal with Hausdorff dimension of 2, and the fat Menger Sponge with Hausdorff dimension of 3. On the other hand, the current existential results in the literature\cite{R2,RR1} lacks their provision for  the highest potential cardinal of beth- two  and a given Lebesgue measure. This issue gets more complex given the fact that fractals with a positive Lebesgue measure must have a positive integer Hausdorff dimension.  \par 

This work presents a parallel existence result for fractals of a given Hausdorff dimension and a positive Lebesgue measure in n-dimensional Euclidean spaces. In light of this aim, this work provides a comprehensive deterministic framework, providing a second proof for the former results and, in particular, generalizes them  in terms of Lebesgue measure and cardinal number.


\section{Preliminaries}
The reader who has studied fractal geometry is well-equipped with the following definitions and key properties of  the Uniform Cantor sets, Topological Dimension and the Hausdorff Dimension. Henceforth, in this paper we consider the n-dimensional Euclidean space $\mathbb{R}^n$ with its conventional Euclidean metric and the Lebesgue measure of $\lambda_{n}(n\in\mathbb{N})$. 

\subsection{The Uniform Cantor Set}
The Uniform Cantor set, or Smith-Volterra-Cantor set, or the fat Cantor set, was introduced in a series of publications by Henry Smith in 1875, Vito Volterra in 1881, and George Cantor in 1883, respectively. To construct the uniform Cantor set, let $\{\beta_n \}_{n\in\mathbb{N}}$-called the removing sequence- be a sequence of positive numbers with $\sum_{n\in\mathbb{N}} \beta_n=1-l\in(0,1].$ Take $s\in \mathbb{N},$ and set $I_{s,0}=[0,1].$ Define $I_{s,n} (n\in\mathbb{N})$ recursively by removing $s$ symmetrically located open intervals with equal length of $\frac{\beta_n}{s(s+1)^{n-1}}$ from middle of each interval in $I_{s,n-1}$. Then, $I_{s,n}$ is the union of $(s+1)^n$ disjoint closed intervals $I_{s,n}^{(j)}$ of equal length $\delta_n=\frac{1-\sum_{k=1}^{n}\beta_k}{(s+1)^n}$;  furthermore, the sequence  $\{I_{s,n}\}_{n\in\mathbb{N}_0}$ is decreasing sequence with finite intersection property in the compact space $[0,1].$  
 
\begin{defi}
\label{Definition2.1}
The Uniform Cantor set of order $s\geq 1$ and Lebesgue measure $l\geq 0$ associated with the removing sequence $\{\beta_n \}_{n}$ is defined as:
\begin{equation}
\label{one}
C_{\{\beta_n \}_{n}}(s,l)=\cap_{n=0}^{\infty} I_{s,n}.
\end{equation}
\end{defi}
It is well known that  the Uniform Cantor set $C_{\{\beta_n \}_{n}}(s,l)$ is nowhere dense, totally disconnected, perfect and uncountable \cite{R2}. Next, by definition, $\lambda_{1}([0,1]-I_{s,n})=\sum_{k=1}^{n}s(s+1)^{k-1}(\frac{\beta_k}{s(s+1)^{k-1}})=\sum_{k=1}^{n}\beta_k (n\in\mathbb{N}).$ Consequently, the Lebesgue measure of the Uniform Cantor set $C_{\{\beta_n \}_{n}}(s,l)$ is given by:
\begin{eqnarray}
\label{two}
\lambda_{1} (C_{\{\beta_n \}_{n}}(s,l)) &=& 1- \lambda_{1} ([0,1]-C_{\{\beta_n \}_{n}}(s,l)) = 1- \lambda_{1} (\cup_{n\in\mathbb{N}}([0,1]-I_{s,n}))\nonumber\\
&=& 1-\lim_{n\rightarrow\infty} \lambda_{1}([0,1]-I_{s,n})=1-\sum_{n\in\mathbb{N}}\beta_n=l.
\end{eqnarray}
A prominent family of the Uniform Cantor sets in the literature \cite{R2,R3,R4,R5} is obtained for the case of 
\begin{equation}
\label{three}
\beta_n=\beta_n(s,\beta,l)=((s+1)\beta)^{n-1}(1-(s+1)\beta)(1-l)
\end{equation}
where $s\geq 1, 0<\beta<\frac{1}{s+1}$ and $0\leq l<1.$ In particular, the ordinary middle third Cantor set is obtained whenever $s=1, \beta=\frac{1}{3},$ and $l=0.$ It is simultaneously the first term of two distictive sequences of the Uniform Cantor sets: one  including $s\geq 1,\ \beta=\beta(s)=\frac{1}{2s+1}(s\geq 1),\ l=0 $ \cite{R2}; and, another including $s=1,\ \beta=\beta(s^*)=\frac{1}{s^*+2}(s^* \geq 1),\ l=0,$ \cite{R3}. 

\subsection{Topological Dimension}
The Topological dimension considered here is the  Urysohn-Menger small inductive dimension. It is defined inductively by setting $dim_{ind}(\phi)=-1.$ We then say that for given $C\subseteq \mathbb{R}^n,$ we have $dim_{ind}(C)\leq k$ whenever there is a base $(U_i)_{i\in I}$ of open sets of $C$ such that   $dim_{ind}(U_i)\leq k-1 (i\in I).$ We then say $dim_{ind}(C)=k$ whenever $dim_{ind}(C)\leq k$ but $dim_{ind}(C)\nleq k-1 (k\in\mathbb{N}_0),$ \cite{R5}.\par 
For the case of a zero-topological dimension, we may consider a more informative definition based on the idea of clopen (i.e., simultaneously open and closed) sets as follows \cite{R5}:  
\begin{defi}
\label{Definition2.2}
The space  $C\subseteq \mathbb{R}^n$ has zero topological dimension (i.e. $dim_{ind}(C)=0$) whenever every finite open cover $\{U_i \}_{i\in I} (|I|<\infty)$ of $C$ has a finite refinement  $\{V_i \}_{i\in J} (|J|\leq |I|<\infty)$ that is a clopen partition of $C.$
\end{defi}
Using above definition, it has been shown that for the case of middle third Cantor set $C=C_{\{\frac{2^n}{2.3^n}\}_n}(1,0)$ we have $dim_{ind}(C)=0,$ \cite{R5}. A straightforward generalization of the proof shows that for the uniform Cantor set $C=C_{\{\beta_n \}_{n}}(s,l)$ we also have the same conclusion. The following theorem summarizes some prominent properties of the small inductive dimension required in this paper\cite{R5}:

\begin{thm}
\label{Theorem2.3}
Let  $\{C_i\}_{i\in I}$ be a countable family of subsets of $\mathbb{R}^{n}\ (n \geq 1).$ Then:\newline\\
(i)\ $dim_{ind}(C_i)\in \{-1,0,\cdots,n \}$ where $i\in I,$\\
(ii)\ $dim_{ind} (C_i)\leq dim_{ind}(C_j)$ where $C_i\subseteq C_j (i\neq j),$\\
(iii)\ $dim_{ind} (cC_i+d)=dim_{ind}(C_i)$ where $c,d\in \mathbb{R}^{+}$ and $i\in I,$\\
(iv)\ $dim_{ind} (\cup_{i\in I} C_i)\leq k$  for some fixed $k\in\mathbb{N}_0$ whenever we have $dim_{ind} (C_i)\leq k (i\in I).$\\
(v)\ $dim_{ind} (\prod_{i\in I} C_i)\leq\sum_{i\in I} dim_{ind}(C_i)$ whenever $I$ is finite. 
\end{thm}

As a corollary of this theorem, for the case of a finite family of uniform Cantor sets $\{ C_i\}_{i\in I},$ it follows that $dim_{ind} (\cup_{i\in I} C_i)= 0$ and $dim_{ind} (\prod_{i\in I} C_i)=0.$

\subsection{Hausdorff Dimension}  
The Hausdorff dimension, or Hausdorff-Besicowitch dimension, is considered as the one of the most prominent dimensions for fractals. It was first introduced by Felix Hausdorff in 1918 and was later  improved in terms of computational techniques by Abram S. Besicovitch. It is definable for any subset of the real line as follows\cite{R4}:  
\begin{defi}
\label{Definition2.4}
Let $C\subseteq \mathbb{R}$ and given $s\geq 0.$ Then, given s-dimensional Hausdorff measure of $C$ by 
\begin{equation}
\label{four}
H^{s}(C)=\lim_{\delta\rightarrow 0} \Big( \inf \{\sum_{i\in\mathbb{N}} |U_i|^{s} |\ \  C\subseteq \underset{i\in \mathbb{N}}{\cup} U_i:  0<|U_i|\leq \delta\} \Big),
\end{equation}
the Hausdorff dimension of $C$ is defined as:
\begin{equation}
\label{five}
dim_{H}(C)=\inf\{s\geq0|H^{s}(C)=0\}.
\end{equation}
\end{defi}

Using this definition, Hausdorff showed that the middle third Cantor set $C=C_{\{\frac{2^n}{2.3^n}\}_n}(1,0)$ has Hausdorff dimension equal to $\frac{log(2)}{log(3)}$. More generally\cite{R5}:
\begin{thm}
\label{Theorem2.5}
The Hausdorff dimension of the Uniform Cantor set  $C=C_{\{\beta_n \}_{n}}(s,l)$ is given by:
\begin{equation}
\label{six}
dim_{H}(C_{\{\beta_n \}_{n}}(s,l))=
\begin{cases}
\frac{log(s+1)}{log(s+1)+\underset{n\rightarrow\infty}{\limsup}(-\frac{log(\beta_n)}{n-1})} &
\textnormal{if $s\geq 1,\ \ l=0,\  \underset{n\in \mathbb{N}}{\inf}(\frac{\beta_n}{1-\sum_{k=1}^{n-1}\beta_k})>0,$}  \\
1
&
\textnormal{if $s\geq 1,\ \ 0<l< 1.$}
\end{cases}
\end{equation}
\end{thm}
As a Corollary, the Hausdorff dimension of the above prominent family of Cantor sets in equation \eqref{three} is given by:
\begin{equation}
\label{seven}
dim_{H}(C_{\{\beta_n(s,\beta,l)\}_{n}}(s,l))=
\begin{cases}
\frac{log(s+1)}{-log(\beta)}
&
\textnormal{if\ \ $s\geq 1,\ 0<\beta<\frac{1}{s+1},\ l=0$}\\
1
&
\textnormal{if\ \ $s\geq 1,\ 0<\beta<\frac{1}{s+1},\ 0<l< 1.$}
\end{cases}
\end{equation}

Next, we consider useful notation for the Uniform Cantor sets with Hausdorff dimension $r>0$ as follows:
\begin{notat}
\label{Notation2.6}
The Linear Transform of the Uniform Cantor set of order $s\geq 1$, Hausdorff dimension $r>0$ and  Lebesgue measure $l\geq 0$   associated with the removing sequence $\{\beta_n \}_{n}$ is denoted as:
\begin{equation}
\label{eight}
F_{r,l,s}=c\times C_{\{\beta_n \}_{n}}(s,l)+d\  \text{for some}\ c>0,d\geq0.
\end{equation}
\end{notat}

The following theorem summarizes some key properties of the Hausdorff Dimension required in this paper\cite{R4}:

\begin{thm}
\label{Theorem2.7}
Let  $\{C_i\}_{i\in I}$ be a countable family of subsets of $\mathbb{R}^{n}\ (n \geq 1).$ Then:\newline\\
(i)\ $0\leq dim_{H}(C_i)\leq n$ where $i\in I,$\\
(ii)\ $dim_{H} (C_i)\leq dim_{H}(C_j)$ where $C_i\subseteq C_j (i\neq j),$\\
(iii)\ $dim_{H} (cC_i+d)=dim_{H}(C_i)$ where $c,d\in \mathbb{R}^{+}$ and $i\in I,$\\
(iv)\ $dim_{H} (\cup_{i\in I} C_i)=\sup_{i\in I}(dim_{H} C_i),$\\
(v)\ $dim_{H} (\prod_{i\in I} C_i)=\sum_{i\in I} dim_{H}(C_i)$ whenever $I$ is finite and one of $C_i$s is a uniform Cantor set. \\
(vi)\ $dim_{H}(C_i)=n (i\in I)$ whenever $\lambda_{n}(C_i)>0 (i\in I).$
\end{thm}

Finally, throughout this paper we refer to the \emph{fractal} in terms of Mandelbrot's definition \cite{R1}:
\begin{defi}
\label{Definition2.8}
A subset $C\subseteq \mathbb{R}^{n}$ with Hausdorff dimension $dim_H (C)$ and the topological inductive dimension $dim_{ind}(C)$ is a fractal whenever:
\begin{equation}
\label{nine}
dim_H (C) > dim_{ind}(C).
\end{equation}
\end{defi}

\subsection{General Cartesian Product Distribution over Unions}
We finish this section with a review of the relationship between union and Cartesian products. As the Cartesian product is distributive over unions \cite{R6}, we have the following general result easily proved by induction on the dimension of the product $n$:

\begin{thm}
\label{Theorem2.9}
Let $I_{j}(1\leq j\leq n)$ be indexing sets and $\{C_{i_j} \}_{i_j \in I_j} (1\leq j\leq n)$ be families of subsets of $\mathbb{R}$ indexed by them, respectively. Then:
\begin{equation}
\label{ten}
\prod_{j=1}^{n} (\bigcup_{i_j\in I_j} C_{i_j})=\bigcup_{(i_1,\cdots,i_n)\in \prod_{j=1}^{n} I_j} (\prod_{j=1}^{n} C_{i_j}),
\end{equation}
where $\bigcup$ denotes the union and $\prod$ denotes the Cartesian product. 
\end{thm}

\section{Main Results}

We generalize the existence Hausdorff Dimension Theorem from fractals with a Lebesgue measure 0 (thin fractals) to those with a non-negative Lebesgue measure (fat fractals) with existence of the higher cardinal of beth- two. The construction process is accomplished in  four stages: (i) Showing the existence of fractals with a plausible Hausdorff dimension and a Lebesgue measure in $\mathbb{R}$; (ii) Expanding the cardinality of fractals in the first stage to the continuum;  (iii) Exending the result in the second stage to the higher dimensional Euclidean spaces $\mathbb{R}^{n} (n>1),$ and, (iv) Generalizing the result in the third stage to the cardinal of beth- two. We begin with the following existence result whose very special case has been stated in \cite{RR1}:
\begin{lem}
\label{Lemma3.1} 
For any real $0<r\leq 1$ and $l\geq 0$ there is a fractal  with the Hausdorff dimension $1_{[0]}(l).r+1_{(0,\infty)}(l)$ and the Lebesgue measure $l$ in $\mathbb{R}.$
\end{lem}

\begin{proof}
We consider three scenarios:\newline\\
(i) $0<r<1, l=0:$\newline
Fix $s_0\in\mathbb{N},$ and consider the family of Uniform Cantor sets $C_{\{\beta_n(s_0,\beta, 0)\}}(s_0,0), $ $ (0<\beta<\frac{1}{s_0+1}).$ Then, their Hausdorff dimension are given by $f_{s_0}(\beta)=\frac{log(s_0+1)}{-log(\beta)},$ $ (0<\beta<\frac{1}{s_0+1}).$ Since $f_{s_0}$ is a continuous increasing function from $(0,\frac{1}{s_0+1})$ onto $(0,1),$ such that $f(0^{+})=0$ and $f(\frac{1}{s_0+1}^{-})=1,$ by an application of the Intermediate Value Theorem, there exists $\beta_0\in(0,\frac{1}{s_0+1})$ such that $f_{s_0}(\beta_0)=r.$ Now, it is enough to consider the following fractal in the interval $[0, 1]:$
\begin{equation}
\label{eleven}
F_{r,0,s_0}=C_{\{\beta_n(s_0,\beta_0, 0)\}}(s_0,0)
\end{equation}
 
\noindent(ii) $r=1, l=0:$\newline
Again, fix $s_0\in\mathbb{N}$ and consider the sequence $r_n=\frac{n}{n+1}\ (n\in\mathbb{N}).$ By part (i), there is a corresponding sequence of fractals $F_{r_n,0,s_0}\ (n\in\mathbb{N})$ such that $dim_{H}(F_{r_n,0,s_0})=r_n, dim_{ind}(F_{r_n,0,s_0})=0,$ and $\lambda_{1}(F_{r_n,0,s_0})=0,\  (n\in\mathbb{N}).$ Take:
\begin{equation}
\label{twelve}
F_{1,0,s_0}^{*}= \underset{n\in \mathbb{N}}{\cup} F_{r_n,0,s_0}  
\end{equation}
Accordingly, two consecutive applications of Theorem \ref{Theorem2.3}(iv) and Theorem\ref{Theorem2.7}(iv) proves the desired result. \newline\\
(iii) $r=1, l>0:$\newline
Finally, for fixed $s_0\in\mathbb{N}$  consider $F_{1,\frac{l}{[l]+1},s_0}=([l]+1)\times C_{\{\beta_n(s_0,\beta_0, \frac{l}{[l]+1})\}}(s_0,\frac{l}{[l]+1})$ with $\beta_n$ given by equation \ref{three}. Hence, two consecutive applications of Theorem \ref{Theorem2.3}(iii) and Theorem\ref{Theorem2.7}(iii) yields the desired result.
\end{proof}
A closer look at the proof of the Lemma \ref{Lemma3.1} and changing values $s_0\in\mathbb{N}$ shows that indeed for any real $0<r\leq 1$ and $l \geq 0$ there are countably infinite fractals  with the Hausdorff dimension $1_{[0]}(l).r+1_{(0,\infty)}(l)$  and the Lebesgue measure $l$ in $\mathbb{R}$. The following Lemma expands this result to the cardinality of the Continuum. 

\begin{lem}
\label{Lemma3.2} 
For any real $0<r\leq 1$ and $l\geq 0$ there is a continuum of  distinctive fractals with the Hausdorff dimension $1_{[0]}(l).r+1_{(0,\infty)}(l)$ and the Lebesgue measure $l$ in $\mathbb{R}.$
\end{lem}

\begin{proof}
We consider three scenarios similar to the proof of Lemma \ref{Lemma3.1}. In each scenario, we partition a given interval (e.g. $[0,1]$) to an infinite union of its shrinking sub-intervals and, using linear transformations, put into them Cantor Fractals.  \newline\\
(i) $0<r<1, l=0:$\newline
By construction, in the part(i) of the proof of Lemma\ref{Lemma3.1}, there is sequence of fractals $\{F_{r,0,s} \}_{s\in\mathbb{N}}$ with $F_{r,0,s}=\frac{1}{2^s}C_{\beta_n(s,\beta_s,0)}+\frac{1}{2^s}$ in the interval $[\frac{1}{2^s}, \frac{1}{2^{s-1}}]$ such that: $dim_{H}(F_{r,0,s})=r, dim_{ind}(F_{r,0,s})=0,$ and $\lambda_{1}(F_{r,0,s})=0\ (s\in\mathbb{N}).$ Let $I\in\mathcal{P}(\mathbb{N})$ be infinite and define:
\begin{equation}
\label{thirteen}
F_{r,0}^{I}= \underset{s\in I}{\cup} F_{r ,0,s }  
\end{equation}
Then, again an application of Theorem \ref{Theorem2.3}(iii,iv) and Theorem \ref{Theorem2.7}(iii,iv) and uncountability of the set of such $I's$ proves the assertion. \newline\\
(ii) $r=1, l=0:$\newline
Again, consider the sequence $r_n=\frac{n}{n+1} (n\in\mathbb{N}).$ By part (i), there is a corresponding  sequence of fractals $F_{r_n,0}^{I}(n\in\mathbb{N})$ and an infinite set $I\in\mathcal{P}(\mathbb{N})$ such that $dim_{H}(F_{r_n,0}^{I})=r_n, dim_{ind}(F_{r_n,0}^{I})=0, $ and $\lambda_{1}(F_{r_n,0}^{I})=0(n\in\mathbb{N}).$ Put :
\begin{equation}
\label{fourteen}
{F^*}_{1,0}^{I}= \underset{n\in\mathbb{N}}{\cup} (\frac{1}{2^n} F_{r_n,0}^{I} +\frac{1}{2^n})  
\end{equation}
Now, by two applications of Theorem \ref{Theorem2.3}(iii,iv) and Theorem\ref{Theorem2.7}(iii,iv)  and uncountability of the set of such $I's$, the assertion follows.\newline\\
(iii) $r=1, l>0:$\newline
By construction, in the part(iii) of the proof of Lemma\ref{Lemma3.1}, there is sequence of fractals $\{F_{1,\frac{l}{[l]+1},s} \}_{s\in\mathbb{N}}$ in the interval $[0, [l]+1]$ such that: $dim_{H}(F_{1,\frac{l}{[l]+1},s})=1, dim_{ind}(F_{1,\frac{l}{[l]+1},s})=0,$ and $\lambda_{1}(F_{1,\frac{l}{[l]+1},s})=l\ (s\in\mathbb{N}).$ Let $F_{1 ,\frac{l}{[l]+1} ,s }^*=\frac{[l]+1}{2^s}F_{1 ,\frac{l}{[l]+1} ,s } +\frac{[l]+1}{2^s} (s\in\mathbb{N}),$ and  $I\in\mathcal{P}(\mathbb{N})$ be infinite and define:
\begin{equation}
\label{fiftheen}
{F^{*}}^{I}_{1,l}=\frac{l}{\lambda_{1}(\underset{s\in I}{\cup} F_{1 ,\frac{l}{[l]+1} ,s }^* )}\times \underset{s\in I}{\cup} F_{1 ,\frac{l}{[l]+1} ,s }^* 
\end{equation}

Finally, by two applications of Theorem \ref{Theorem2.3}(iii,iv) and Theorem\ref{Theorem2.7}(iii,iv)  and uncountability of the set of such $I's$ the assertion follows.
\end{proof}
An investigation into Lemma \ref{Lemma3.2} reveals its limitation to providing fractals with Hausdorff dimension on the range $0<r\leq 1.$ To obtain fractals with Hausdorff dimension $r>1$ we need to consider the higher dimensional Euclidean spaces. Using the idea of the countable unions of n-dimensional Cantor Fractal dusts (as the Cartesian product of the Uniform Cantor sets defined above), our  next result addresses this situation.

\begin{lem}
\label{Lemma3.3}
For any real $r > 0$ and $l\geq 0,$ there is a continuum  of distinctive fractals with the Hausdorff dimension $1_{[0]}(l).r+1_{(0,\infty)}(l).n$  and  Lebesgue measure $l$ in $\mathbb{R}^{n}$ where $(-[-r] \leq n)$.
\end{lem}
 
\begin{proof}
Let $r>0, l\geq 0$ and for given $-[-r]\leq n $ define $r'=\frac{r}{n}\in(0,1].$Then, using equations (\ref{thirteen}-\ref{fiftheen}), we consider the following  scenarios with the Hausdorff dimension $r'$ and Lebesgue measure $l\geq 0$ given different values of $r',l:$

\begin{table}[H]
\caption{Construction of fractals with the given Hausdorff dimension $r'$ and Lebesgue measure $l$ in $\mathbb{R}$\label{Table1} }
\begin{center}
\begin{tabular}{ c|cc } 
\hline
 $(r',l)$                & $l=0$& $l>0$ \\
\hline 
$0<r'<1$ & $F_{r',l}^{I} (Eq.13)$ &
DNE\\ 
$r'=1$   & ${F^{*}}_{r',l}^{I}(Eq.14)$ &
${F^{*}}_{r',l}^{I}(Eq.15) $ \\ 
\hline
\end{tabular}
\end{center}
\end{table}
where $I\in\mathcal{P}(\mathbb{N}) $ is infinite.  Now, we consider n copies from each cell  in Table \ref{Table1} and consider the following Cartesian products:
\begin{table}[H]
\caption{Construction of fractals with given a Hausdorff dimension $r$ and Lebesgue measure $l$ in $\mathbb{R}^{n}$\label{Table2} }
\begin{center}
\begin{tabular}{ c|cc } 
\hline
 $(r,l)$  & $l=0$& $l>0$ \\
\hline 
$r\notin\mathbb{N}$ & $\prod_{j=1}^{n}F_{r',l}^{I_j} $ &  DNE\\ 
$r\in\mathbb{N}$   & $\prod_{j=1}^{n}{F^{*}}_{r',l}^{I_j} $ & $(\frac{1}{l})^{n-1}\times\prod_{j=1}^{n}{F^{*}}_{r',l}^{I_j} $ \\ 
\hline
\end{tabular}
\end{center}
\end{table}

where $I_j\in\mathcal{P}(\mathbb{N}) (1\leq j\leq n) $ are infinite. Finally, considering Table \ref{Table2}, the assertion follows by an application of Theorem \ref{Theorem2.3}(v), Theorem\ref{Theorem2.7}(iii),(v) and Theorem \ref{Theorem2.9};  and that there are  uncountably infinite sets of $I_j\in\mathcal{P}(\mathbb{N}) (1\leq j\leq n).$ 
\end{proof}

The constructed fractals in the Lemma \ref{Lemma3.3}  have two key features as follows:   First, they are centrally asymmetric with respect to the point symmetric of  $\overrightarrow{c}_n=(\frac{1}{2},\cdots,\frac{1}{2})$ where $-[-r]\leq n.$ This is a direct result of the fact that being centrally asymmetric is invariant under Cartesian products and unions; and, the building block of the fractals in the Lemma are centrally asymmetric fat Cantor sets with point symmetric of $\overrightarrow{c}_1= \frac{1}{2}.$ However, the Lemma can be expanded to the symmetric fractals with consideration of transformed symmetric fat Cantor sets. The details of the proof are minor modifications of our proof for the asymmetric case presented here with replacement of fractals $F$ in equations  (\ref{thirteen}-\ref{fiftheen})  with  $\frac{1}{2}((F \cup (-F))+1)$. Second, there is a continuum of them with the same Hausdorff dimension and the same Lebesgue measure in the Euclidean spaces  $\mathbb{R}^{n+1}-\mathbb{R}^{n}\ \ (-[-r]\leq n)$ as those in  $\mathbb{R}^{n}\ \ (-[-r]\leq n).$ Here, $\mathbb{R}^{n}$ is considered isomorphic to the subspace of $\mathbb{R}^{n}\times 0 \subset \mathbb{R}^{n+1}.$ This result is the direct consequence of the construction process in the proof of the Lemma.  Finally, while the Lemma \ref{Lemma3.3}  guarantees at least a continuum of fractals in n-dimensional Euclidean space with given properties, it does not guarantee the existence of the beth- two fractals given the generalized continuum hypothesis \cite{RRR16}.  This lack of precision of the cardinal number  (between either the continuum or beth- two) creates ground for further investigation. Our main result provides the answer.\par 

\begin{thm}
\label{Theorem3.4}
\textbf{(The General Hausdorff Dimension Theorem)} For any real $r > 0$ and $l\geq 0,$ there is a set of beth- two distinctive fractals with the Hausdorff dimension $1_{[0]}(l).r+1_{(0,\infty)}(l).n$  and  Lebesgue measure $l$ in $\mathbb{R}^{n}$ where $(-[-r] \leq n)$.
\end{thm}

\begin{proof}
We accomplish the proof in four steps as follows:\newline\\
\indent First, let $r>0, l\geq 0,$ and fix $n\geq -[-r].$ Then, for the function $h_n(\beta)=\frac{n.log(2)}{-log(\beta)},$ $\ (0<\beta<\frac{1}{2}),$ given $h_n(0^{+})=0$ there is $\beta_{00}\in (0,\frac{1}{2})$ such that $h_n(\beta_{00})<\frac{r}{2}.$ \\ 
\indent Second, take a fractal $F_{00}=(C+[l]+1)^n$ in $\mathbb{R}^n$ where $C=C_{\beta_n}$ with $\beta_n=\beta_n(1,\beta_{00},0)$ $\ (n\geq 1)$ as in equation (\ref{three}).  Then, $\lambda(F_{00})=0,\ \ \dim_{ind}(F_{00})=0,$ and $dim_{H}(F_{00})<\frac{r}{2}.$ Define the power set of $F_{00}$ with cardinal of beth- two as follows: 
\begin{equation}
\label{sixtheen}
\mathcal{F}_1(r,l)= \{ F\subseteq \mathbb{R}^n | F\subseteq F_{00}, dim_{H}(F)<\frac{r}{2}, \lambda(F)=0, dim_{ind}(F)=0  \}. 
\end{equation}
 \indent Third, consider the family of fractals given by Lemma \ref{Lemma3.3}:
\begin{equation}
\label{seventheen}
 \mathcal{F}_2(r,l)=\{ F\subseteq \mathbb{R}^n| dim_{H}(F)=1_{[0]}(l).r+1_{(0,\infty)}(l).n,\ \lambda(F)=l, dim_{ind}(F)=0 \}. 
\end{equation}
 \indent Fourth, given two families in equations (\ref{sixtheen}),(\ref{seventheen})  define the following family of fractals with cardinal of beth- two: 
\begin{equation}
\mathcal{F}(r,l)=\{F_1\cup F_2| F_1\in \mathcal{F}_1(r,l), F_2\in \mathcal{F}_2(r,l)\}. 
\end{equation}
 \indent Finally, let $F\in\mathcal{F}(r,l).$ Then, by construction, there are $F_i\in\mathcal{F}_i(i=1,2)$ such that: $F=F_1\cup F_2.$ Accordingly, by another application of Theorem \ref{Theorem2.3}(iv), and Theorem \ref{Theorem2.7}(iv):
\begin{eqnarray}
dim_{H}(F)&=&\max(dim_{H}(F_1),dim_{H}(F_2))\nonumber\\
&=&dim_{H}(F_2)=1_{[0]}(l).r+1_{(0,\infty)}(l).n,
\\
\lambda(F)&=&\lambda(F_2)=l,\\
dim_{ind}(F)&=&0.
\end{eqnarray}
This completes the proof.
\end{proof}

The assertion of Theorem \ref{Theorem3.4} and its proof methodology have two immediate consequences: First, in terms of having the cardinality of beth-two, the theorem's assertion and an application of the Cantor-Schroder-Bernstein theorem \cite{RRRR17} put the set of virtual fractals of $\mathbb{R}^n$  the set of virtual fractals of $\mathbb{R}^n$ in the same category of mathematical sets such as the sigma-algebra of Lebesgue measurable sets in $\mathbb{R}^n,$  the power set of $\mathbb{R}^n,$  the set of functions from $\mathbb{R}^n$  to $\mathbb{R}^n,$ and, the Stone-Cech compactification of $\mathbb{R}^n.$  Second, the theorem's proof can be modified to show that the cardinality of the set of non-fractals in $\mathbb{R}^n$  is beth-two as well. To see this, let $n \geq 1, C$ be the standard middle third Cantor set, and $I^n$ be the unit cube in $\mathbb{R}^n$. Define the family $\mathcal{F}(n)=\{I^n \cup P_{C,n}| P_{C,n}\subseteq (C+1)^n \}$. Then, for any  $F\in\mathcal{F}(n)$ we have $dim_{ind}(F)=n=dim_{H}(F).$ Hence, by Mandelbrot's definition \ref{Definition2.8},$F$ is a non-fractal. Consequently, $\mathcal{F}(n)$ is a subset of the set of non-fractals in $\mathbb{R}^n$. Accordingly, the assertion follows from $\mathcal{F}(n)$ having the cardinality of beth-two and applying the Cantor-Schroder-Bernstein theorem \cite{RRRR17}.

\section{Discussion}
  
This work presented an existence theorem for fractals of a given Hausdorff dimension and a Lebesgue measure by considering finite Cartesian products of the countable unions of Uniform Cantor sets with plausible Lebesgue measure. In addition, it generalized the former existence theorem in terms of the Lebesgue measure  and cardinal number. \par 
This work's contributions to the fractal geometry literature covers two perspectives: First, it highlights the advantage of the Cantor sets to other well-known classical fractals in showing existence of fractals with any Hausdorff dimension and Lebesgue measure. Other prominent fractals lack this feature given not being defined in the one dimensional Euclidean space $\mathbb{R}$ and having Hausdorff dimension larger than 1. Examples include the Sierpinski triangle, Takagi curve, Julia set, Triflake, Koch curve, and Apollonian gasket. Second, the existence theorem is equipped with constructive proof (versus pure existence proof) presenting real fractals for given Hausdorff dimension and the Lebesgue measure \cite{R7}. This key feature helps us to explore other properties of the constructed fractals yielding to more comprehensive information on them. \par 
There are some limitations in this work that create four new lines of research for the interested reader. First, we considered only Mandelbrot's strict mathematical definition to present the existence result. However, there are other agreed-upon key descriptive features in the definition of fractal that need to be considered. Examples of these characteristics are self-similarity type(exact, quasi, statistical), fine structure, irregularity(local, global), and the recursive definition\cite{R4}. Second, while we considered only Hausdorff dimension as the index of fractal dimension in the Mandelbrot's definition, the existence case for the other indices of fractal dimension remains to be investigated. For instance, we can consider the Renyi dimension (with special cases, such as the Minkowski dimension \cite{R4}, Information dimension \cite{R8}, Correlation dimension \cite{R9}), the Higuchi dimension \cite{R10}, the Lyapunov dimension \cite{R11}, Packing dimension\cite{R12}, the Assouad dimension\cite{R13} and the generalization of the Hausdorff dimension\cite{R14}. Third, a more rigorous and comprehensive method is to investigate the existence problem of fractals for the generalized fractal space equipped with a fractal structure and the generalized fractal dimension \cite{R15}. Finally, the existence result in this work is limited to deterministic fractals constructed by their associated deterministic recursive processes. However, its validity for the more general random fractals remains an open question. We summarize the above points as the following set of open problems: 

\subsection*{ Open Problems}
 Given the n-dimensional Euclidean space $\mathbb{R}^n (n\geq 1).$ \newline\\
 (1) Does the precise cardinality of the set of all distinctive fractals vary by the applied fractal dimension ? \newline\\
 (2) Does the precise cardinality of the set of all distinctive fractals depend on the fractal structure ? \newline\\
 (3) Does the precise cardinality of the set of all distinctive fractals depend on the deterministic or random nature of the fractal? 

\section{Conclusion} There has been a growing interest in fractal dimension and its calculations from different methods since the early 1900s without retrospective exploration of the existence of fractals for a given dimension. This work addressed the problem in part by presenting a constructive existence result from deterministic viewpoint, paving the way for exploring a more generalized random perspective.


\end{document}